\documentclass[a4paper,11pt]{article}
\usepackage[active]{srcltx}
\usepackage{array}
\usepackage{theorem}
\usepackage{amsmath,amscd,amssymb}
\usepackage{latexsym}
\usepackage{epsfig}
\usepackage{xy}

\theorembodyfont{\sl}

\newtheorem{lemma}{Lemma}[section]

\newtheorem{proposition}[lemma]{Proposition}
\newtheorem{theorem}[lemma]{Theorem}
\newtheorem{corollary}[lemma]{Corollary}

\newtheorem{definition}[lemma]{Definition}

\newcommand{\CC}{\mathbb C}

\newcommand{\HH}{\mathbb H}

\newcommand{\NN}{\mathbb N}

\newcommand{\PP}{\mathbb P}
\newcommand{\QQ}{\mathbb Q}

\newcommand{\ZZ}{\mathbb Z}

\newcommand{\cD}{\mathcal D}

\newcommand{\cF}{\mathcal F}

\newcommand{\cH}{\mathcal H}

\newcommand{\cM}{\mathcal M}

\newcommand{\cS}{\mathcal S}

\renewcommand{\Tilde}{\widetilde}

\renewcommand{\Bar}{\overline}

\newcommand{\Sp}{\mathop{\mathrm {Sp}}\nolimits}
\newcommand{\PSp}{\mathop{\mathrm {PSp}}\nolimits}

\newcommand{\SL}{\mathop{\mathrm {SL}}\nolimits}
\newcommand{\SO}{\mathop{\mathrm {SO}}\nolimits}

\newcommand{\Orth}{\mathop{\null\mathrm {O}}\nolimits}

\newcommand{\TS}{\mathop{\Tilde{\mathrm {SO}}^+}\nolimits}
\newcommand{\TO}{\mathop{\Tilde{\mathrm {O}}^+}\nolimits}

\newcommand{\Lift}{\mathop{\mathrm {Lift}}\nolimits}

\newcommand{\Pic}{\mathop{\mathrm {Pic}}\nolimits}

\renewcommand{\Im}{\mathop{\mathrm {Im}}\nolimits}

\newcommand{\rank}{\mathop{\mathrm {rank}}\nolimits}

\newcommand{\sign}{\mathop{\mathrm {sign}}\nolimits}

\newcommand{\id}{\mathop{\mathrm {id}}\nolimits}

\newcommand{\supp}{\mathop{\mathrm {supp}}\nolimits}

\newcommand{\gz}{\mathfrak z}
\newcommand{\gG}{\mathfrak G}

\renewcommand{\div}{\mathop{\mathrm {div}}\nolimits}
\newcommand{\Bdiv}{\mathop{\mathrm {R.div}}\nolimits}
\newcommand{\Kthree}{\mathop{\mathrm {K3}}\nolimits}
\newcommand{\qedsymbol}{\mbox{$\Box$}}
\newcommand{\qed}{\unskip\nobreak\hfil\penalty50\hskip1em\hbox{}\nobreak
\hfill\qedsymbol\parfillskip=0pt\finalhyphendemerits=0}
\newenvironment{proof}{\begin{ProofwCaption}{Proof}}{\end{ProofwCaption}}
\newenvironment{ProofwCaption}[1]
 {\addvspace\theorempreskipamount \noindent{\it #1.}\rm}
 {\qed \par \addvspace\theorempostskipamount}

\begin{document}

\title{Reflective modular forms in  algebraic geometry}
\author{V.~Gritsenko}
\date{}
\maketitle
\begin{abstract}
We prove that the existence of a strongly reflective modular form of a large  weight
implies that the Kodaira dimension of the corresponding modular variety
is negative or, in some special case, it is equal to zero.
Using the Jacobi lifting we construct three  towers of strongly reflective
modular forms with the simplest possible divisor.
In particular we obtain a Jacobi lifting  construction of the Borcherds-Enriques
modular  form $\Phi_4$ and Jacobi liftings of automorphic discriminants
of the K\"ahler  moduli  of Del Pezzo surfaces constructed recently by Yoshikawa.
We obtain also  three modular varieties of dimension $4$, $6$ and $7$
of Kodaira dimension $0$.
\end{abstract}

\bigskip
\section{Introduction}
\label{sec0}
A reflective modular form is a  modular form on an orthogonal group of
type $\Orth(2,n)$ whose divisor is determined by reflections.
A strongly reflective form
that vanishes of order one along the reflective divisors
is the denominator function
of a Lorentzian Kac-Moody (super) Lie algebra of Borcherds type.
For example the famous Borcherds form $\Phi_{12}$ in $26$ variables (see \cite{B1})
defines the Fake Monster Lie algebra.

Reflective  modular forms are very rare. Some of them have  geometric interpretation
as automorphic discriminants of some moduli spaces, for example of moduli
spaces of lattice polarised $\Kthree$ surfaces (see \cite{GN5}).
The first such  example was the  Borcherds--Enriques  form
$\Phi_4$ (see \cite{B2}). This  strongly reflective form is the automorphic discriminant of the
moduli space of Enriques surfaces and it is the  denominator function
of the fake monster superalgebra. In 2009  K.-I.~Yoshikawa constructed the automorphic
discriminant $\Phi_V$ of the K\"ahler  moduli  of  a Del Pezzo surface
$V$ of $1\le \deg V\le 9$.
These functions are also related to the analytic torsion of special Calabi--Yau
threefolds (see \cite{Y}). The corresponding Borcherds superalgebras  were predicted
in the conjecture of Harvey--Moore (see \cite[\S 7]{HM}).
We note that  the generators and relations of
Lorentzian Kac--Moody (super) Lie algebras of Borcherds type
are defined by the Fourier expansion of a reflective modular form
at a  zero-dimensional cusp (see \cite{GN1} and \cite{GN3}).
All  modular forms mentioned above were constructed
as Borcherds automorphic products which gives us the multiplicities
of the positive roots.

The quasi pull-backs of the strongly reflective  form $\Phi_{12}$  help us to prove the general
type  of some modular varieties of orthogonal type.
See \cite{GHS1} where we proved  that the moduli space of polarised $\Kthree$ surfaces
of degree $2d$ is of the  maximal Kodaira dimension if $d>61$.
In \S 1 of this paper we give a new geometric definition based on the results of
\cite{GHS1} of
the reflective modular forms as modular forms with a small divisor.
It gives us a new interesting application of reflective modular forms
which is quite  opposite to the results of \cite{GHS1}--\cite{GHS2}.
In Theorem \ref{thm-kdim} we prove that the existence of a strongly reflective modular form
of a large  weight implies that the Kodaira dimension of the corresponding modular variety
is negative or, in some special case, it is equal to zero.
In \S 2 we give three new examples of  modular varieties  of orthogonal type
of dimension $4$, $6$ and $7$ of Kodaira dimension $0$ (varieties of Calabi--Yau type)
and we hope to consider more examples in the near future.

The geometric examples of \S 2 are based on the three towers of strongly reflective
modular forms which we construct in \S3--\S5 with the help of Jacobi lifting.
It is a rather surprising fact that we obtain very simple  Jacobi lifting
constructions of the Borcherds form $\Phi_4$ and of the Yoshikawa functions $\Phi_V$.
These modular forms constitute the  $D_8$-tower of the Jacobi liftings
(see \S 3). In particular we obtain simple formulae for the Fourier expansions
of $\Phi_V$, i.e. the explicit generating formulae for the imaginary simple
(super) roots of the corresponding Borcherds superalgebras.
In \S 4 and \S 5 we present the towers of the strongly reflective modular  forms
based on the modular forms of singular weight for the root systems $3A_2$
and $4A_1$. The $D_8$-, $3A_2$- and $4A_1$-towers of the Jacobi liftings give
us $15=8+3+4$ strongly reflective modular forms.
We note that the last function in  the $4A_1$-tower is
the Siegel modular form $\Delta_5$ which is
the square root of the Igusa modular form of weight $10$.

In the conclusion  we formulate a conjecture about  strongly reflective
modular forms  similar to the modular forms considered in this paper.

\section{Modular varieties of orthogonal type and reflective modular forms}
\label{sec1}

We start with the general set-up.
Let $L$ be an even integral   lattice with a
quadratic form of signature $(2,n)$ and let
$$
  \cD(L)=\{[Z] \in \PP(L\otimes \CC) \mid
  (Z,Z)=0,\ (Z,\Bar Z)>0\}^+
$$
be the associated $n$-dimensional classical Hermitian domain  of type $IV$
(here $+$ denotes one of its two connected components).
We denote by $\Orth^+(L)$ the index $2$ subgroup of the integral orthogonal group $\Orth(L)$
preserving $\cD(L)$.
For any $v\in L\otimes \QQ$ such that $v^2=(v,v)<0$ we define  the {\it rational quadratic divisor}
$$
\cD_v=\cD_v(L)=\{[Z] \in \cD(L)\mid (Z,v)=0\}\cong \cD(v^\perp_L)
$$
where $v^\perp_L$ is an even integral  lattice of signature $(2,n-1)$.
If $\Gamma<\Orth^+(L)$ is of finite index we
define  the corresponding {\it modular variety}
$$
\cF_L(\Gamma)=\Gamma\backslash \cD(L),
$$
which is a quasi-projective variety of dimension $n$.

The important examples of modular varieties of orthogonal type are

\noindent
a) the moduli spaces of polarised $\Kthree$ surfaces;

\noindent
b) the moduli spaces of lattice polarised $\Kthree$ surfaces
(the dimension of such a moduli space is smaller than  $19$);

\noindent
c) the moduli spaces of polarised Abelian or Kummer  surfaces;

\noindent
d) the moduli space of Enriques surfaces;

\noindent
e) the periodic domains of polarised  irreducible symplectic varieties
(the dimension of a modular variety of this type  is equal to $4$, $5$, $20$ or $21$).

One of the main tools in the study of the geometry of modular varieties
is the theory of modular forms with respect to an orthogonal group.
In the next definition we bear in mind
Koecher's principle (see \cite{B1}, \cite{Bai}).
\begin{definition}\label{def-mf} Let $\sign(L)=(2,n)$ with $n\ge 3$.
A  modular form of weight $k$ and character $\chi\colon \Gamma\to
\CC^*$ with respect to $\Gamma$  is a holomorphic function $F\colon\cD(L)^\bullet\to \CC$
on the affine cone $\cD(L)^\bullet$ over $\cD(L)$ such that
$$
  F(tZ)=t^{-k}F(Z)\quad \forall\,t\in \CC^*,
$$
$$
  F(gZ)=\chi(g)F(Z)\quad  \forall\,g\in \Gamma.
$$
A modular form is called a  cusp form if it vanishes at every cusp
(a boundary component of the Baily--Borel compactification of $\cF_L(\Gamma)$).
\end{definition}

We denote the linear spaces of modular and cusp forms of weight $k$ and
character (of finite order) $\chi$  by $M_k(\Gamma,\chi)$ and
$S_k(\Gamma,\chi)$ respectively.
If $M_k(\Gamma,\chi)$ is nonzero then one knows that $k\ge (n-2)/2$ (see \cite{G1}--\cite{G2}).
The minimal weight $k=(n-2)/2$ is called {\it singular}.
The weight $k=n=\dim(\cF_L(\Gamma))$  is called {\it canonical} because according to  Freitag's criterion
$$
S_n(\Gamma, \det)\cong H^{0}\bigl(\Bar{\cF}_L(\Gamma), \Omega(\Bar{\cF}_L(\Gamma))\bigr),
$$
where $\Bar{\cF}_L(\Gamma))$ is a smooth compact model of the modular variety
$\Bar{\cF}_L(\Gamma)$ and $\Omega(\Bar{\cF}_L(\Gamma))$ is the sheaf of  canonical
differential forms
(see \cite[Hilfssatz 2.1, Kap. 3]{F}).
We say that
\begin{equation}\label{wt-sb}
{\rm the \ weight\ } \ k  {\rm\ \  is\ small \  if \ \ } k<n\  {\rm\  or \   big\  if\ \ } k\ge n.
\end{equation}
For applications, the most important  subgroups of $\Orth^+(L)$ are
the stable orthogonal groups
\begin{equation*}
\Tilde{\Orth}^+(L) =\{g\in \Orth^+(L)\mid g|_{L^\vee/L}=\id\},\quad
\Tilde{\SO}^+(L)=\SO(L)\cap \Tilde{\Orth}^+(L)
\end{equation*}
where $L^\vee$ is the dual lattice of $L$.
If the lattice $L$ contains two orthogonal copies of the hyperbolic plane
$U\cong \left(\begin{smallmatrix}
0&1\\1&0
\end{smallmatrix}\right)$
(the even unimodular lattice of signature $(1,1)$) and its reduction modulo
$2$ (resp. $3$) is  of rank at least $6$ (resp. $5$) then  $\Tilde{\Orth}^+(L)$ has only one
non-trivial character $\det$ (see \cite{GHS3}).
\smallskip

For any  non isotropic $r\in L$ we denote by $\sigma_r$  reflection with respect to $r$
$$
\sigma_r(l)=l-\frac{2(l,r)}{(r,r)}r\in  \Orth(L\otimes \QQ).
$$
This is an element of  $\Orth^+(L\otimes \QQ)$ if and only if $(r,r)<0$.
If $r^2=-2$, then  $\sigma_r(l)\in \Tilde{\Orth}^+(L)$.
In general, the ramification divisor of $\cF_L(\Tilde{\Orth}^+(L))$ is larger than
the  union of the rational quadratic divisors $\cD_r(L)$ defined by  $(-2)$-roots in  $L$.

\begin{definition}\label{def-reflf}
A modular form $F\in M_k(\Gamma,\chi)$ is called {\bf reflective} if
\begin{equation}\label{eq-reflf}
\supp (\div F) \subset \bigcup_{\substack{\pm r\in L \vspace{1\jot}\\
r\  {\rm is\   primitive}\vspace{1\jot}\\
\sigma_r\in \Gamma\text{ or }-\sigma_r\in \Gamma}} \cD_r(L)
=\Bdiv(\pi_\Gamma).
\end{equation}
We call $F$ {\bf strongly reflective} if the multiplicity of any irreducible component
of $\div F$ is equal to one.
\end{definition}

We note that $\cD_r(L)=\cD_{-r}(L)$.
In the  definition of  reflective modular forms in  \cite{GN4}
only the first condition $\sigma_r\in \Gamma$ was considered.
The present definition  is explained by the following result
proved  in \cite[Corollary 2.13]{GHS1}
\begin{proposition}\label{prop-br-div}
The ramification divisors of the modular projection
$$
\pi_\Gamma: \cD(L)\to \Gamma\backslash \cD(L)
$$
are induced by elements $g\in \Gamma$ such that $g$ or $-g$
is a reflection with respect to a vector in $L$.
\end{proposition}
According to the last proposition the union of the rational
quadratic divisors in the right hand side of (\ref{eq-reflf})
is the ramification divisor $\Bdiv(\pi_\Gamma)$ of the modular projection $\pi_\Gamma$.
\smallskip

\noindent
{\bf Example 1} {\it The Borcherds modular form $\Phi_{12}$}.
The most famous  example of  a  strongly reflective modular
form is $\Phi_{12}\in M_{12}(O^+(II_{2,26}), \det)$ (see \cite{B1}).
This is the unique modular form of singular weight $12$ with character $\det$
with respect to the orthogonal group  $O^+(II_{2,26})$ of the even unimodular lattice
of signature $(2,26)$. The form $\Phi_{12}$ is the Kac-Weyl-Borcherds denominator function
of the Fake Monster Lie algebra.
We  expect only finite number  of reflective modular forms
(see \cite{GN3}). The Borcherds automorphic products give us
some number of interesting examples of strongly reflective modular forms
(see \cite{B1}--\cite{B4}).
Note that for a large class of integral lattices any reflective modular form has
a Borcherds product according to  \cite{Br}.

As we mentioned above, if the rank of the quadratic  lattice is smaller or equal to $19$
then one can interpreted the modular varieties of orthogonal type as
moduli spaces of lattice polarised  $\Kthree$ surfaces.
The stable locus of the  reflections of the  integral orthogonal
group is  related to the special singular $\Kthree$ surfaces. It gives us an interpretation
of reflective modular forms as   {\it automorphic discriminants} of these moduli spaces
(see \cite{GN5}).
Moreover, if the modular form is strongly reflective, then the Lorentzian Kac--Moody algebra
determined by the automorphic discriminant can be considered
as a variant of the arithmetic mirror symmetry
for these $\Kthree$ surfaces (see \cite{GN6}).
We remark also that the reflective modular forms of type $\Phi_{12}$  of singular weight
with respect to congruence subgroups of $\SL_2(\ZZ)$
were classified by N. ~Scheithauer (see \cite{Sch}).
\smallskip

\noindent
{\bf Example 2} {\it Igusa modular forms.}
We can apply Definition \ref{def-reflf} to Siegel modular forms of genus $2$
because $\PSp_2(\ZZ)$ is isomorphic to
$\SO^+(L(A_1))$ where $L(A_1)=2U\oplus A_1(-1)=2U\oplus \langle -2\rangle$.
The Siegel modular form of odd weight
$\Delta_{35}\in S_{35}(\Sp_2(\ZZ))$
and the product of the ten  even theta-constants
$
\Delta_{5}\in S_5(\Sp_2(\ZZ), \chi_2)
$
($\chi_2$ is a character of order $2$) are strongly reflective (see \cite{GN1}--\cite{GN2}).
One more classical example is  the ``most odd" even Siegel theta-constant $\Delta_{1/2}$
which is a modular form of weight $1/2$ with respect to the paramodular group
$\Gamma_4$. These examples are  part of the classification of all reflective forms for
the maximal lattices of signature $(2,3)$ in  \cite{GN4}.
Moreover, $\Delta_5$ and $\Delta_{1/2}$ are examples of modular forms
with the simplest divisor (see \cite{CG}).
\smallskip

\noindent
{\bf Remark 1.} {\it Modular forms of canonical weight.}
Let $\sign(L)=(2,n)$. We consider
$F\in M_n(\Gamma,\det)$.
If $\sigma_r\in \Gamma$, then
$F(\sigma_r(Z))=-F(Z)$. Hence $F$ vanishes along $\cD_r$.
If $-\sigma_r\in \Gamma$, then
$$
(-1)^nF(\sigma_r(Z))=F((-\sigma_r)(Z))=\det (-\sigma_r)F(Z)=(-1)^{n+1}F(Z)
$$
and $F$ also vanishes along $\cD_r$.
Therefore  any $\Gamma$-modular form of canonical weight with character $\det$
vanishes along $\Bdiv(\pi_\Gamma)$.
\smallskip

\noindent
{\bf Remark 2.}
{\it Modular forms with small or  big divisor.}
According to the definition above a modular form $F\in M_k(\Gamma,\chi)$
is strongly reflective if and only if
$$
\div F \le \Bdiv(\pi_\Gamma),
$$
i.e.
{\it the divisor of a strongly reflective modular form is small}.
We say that the divisor of a modular form $F\in M_k(\Gamma,\chi)$
is {\it big} if
$$
\div F  \ge  \Bdiv(\pi_\Gamma).
$$
\noindent

The role of modular forms of small weight
(see (\ref{wt-sb})) with a big divisor
in the birational geometry of moduli spaces  was clarified in \cite{GHS1}:
the modular variety $\cF_L(\Gamma)$ is of general type
if there exists a cusp form of small weight with a big divisor.
More exactly, we proved the following theorem called
{\it low weight cusp form trick}\,:
\begin{theorem}\label{thm-gt}
Let $n\ge 9$. The modular variety $\cF_L(\Gamma)$ is of general type
if there exists   $F\in S_k(\Gamma,\det^{\varepsilon})$ ($\varepsilon =0$ or $1$)
of small weight $k<n$ such that
$\div F\ge \Bdiv(\pi_\Gamma)$.
\end{theorem}
This is a particular case of \cite[Theorem 1.1]{GHS1}.
We applied this theorem in order to prove  that  the moduli spaces
of polarised $\Kthree$ surfaces
and the moduli spaces  of polarised
symplectic varieties deformationally equivalent to ${\rm Hilb}^2(\Kthree)$
or to $10$-dimensional symplectic  O'Grady varieties
are of general type (see \cite{GHS1}--\cite{GHS2}).

In this paper we give a new  application of strongly  reflective modular forms,
which is quite opposite to Theorem \ref{thm-gt}. Namely,
the Kodaira dimension of the modular variety $\cF_L(\Gamma)$ is equal to $-\infty$
if there exists a modular  form of big weight with a small  divisor.
More exactly we have
\begin{theorem}\label{thm-kdim}
Let $\sign(L)=(2,n)$ and $n\ge 3$.
Let $F_k\in M_k(\Gamma, \chi)$ be  a strongly reflective modular form
of weight $k$ and character $\chi$ where $\Gamma<\Orth^+(L)$ is of finite index.
By $\kappa(X)$ we denote the Kodaira dimension of $X$.
Then
$$
\kappa(\Gamma\backslash \cD(L))=-\infty
$$
if $k>n$, or $k=n$ and $F_k$ is not a cusp form.
If $k=n$ and $F$ is a cusp form
then
$$
\kappa(\Gamma_\chi\backslash \cD(L))=0,
$$
where $\Gamma_\chi=\ker(\chi\cdot \det)$ is a subgroup of $\Gamma$.
\end{theorem}
\begin{proof}
To prove the first identity of the theorem  we have to show that there are no
pluricanonical differential forms on $\Bar{\cF}_L(\Gamma)$.
Any such differential form can be obtained using a modular form
(see \cite{AMRT} where weight $1$ corresponds to weight $n$ in our definition
of modular forms).
Suppose that $F_{nm}\in M_{nm}(\Gamma, \det^m)$. We may realize $\cD(L)$
as a tube domain   by choosing a $0$-dimensional cusp.
In the corresponding affine coordinates of this tube domain
we take  a holomorphic volume element $dZ$ on $\cD(L)$. Then
the differential form $F_{nm}\,(dZ)^m$ is $\Gamma$-invariant.
Therefore it determines a section of the pluricanonical bundle
$\Omega(\Bar{\cF}_L(\Gamma))^{\otimes m}$ over  a smooth open part of  the modular
variety away from the branch locus of $\pi\colon\cD(L)\to \cF_L(\Gamma)$ and the cusps
(see \cite[Chapter 4]{AMRT} and \cite{F}).
There are three kinds of obstructions to extending $F_{nm}\,(dZ)^m$
to a global section of $\Omega(\Bar{\cF}_L(\Gamma))^{\otimes m}$, namely,
there are elliptic obstructions, arising because of singularities
given by elliptic fixed points of the action of $\Gamma$;
cusp obstructions, arising from divisors at infinity;
and reflective obstructions, arising from the ramification divisor in $\cD(L)$.
The ramification divisor is defined by $\pm$ reflections in $\Gamma$ according
to Proposition \ref{prop-br-div}. Therefore if $F_{nm}$ determines a global section
then $F_{nm}$  has zeroes of order at least  $m$ on $\Bdiv(\pi_\Gamma)$.
The modular form $F_k\in M_k(\Gamma, \chi)$ is strongly reflective of weight $k\ge n$
hence $F_{nm}/F_k^m$ is a holomorphic modular form of weight $m(n-k)\le 0$.
According to Koecher's principle (see \cite{Bai} and \cite{F})
this function is  constant. We have that $F_{nm}\equiv 0$ if $k>n$
or $F_{nm}=C\cdot F_n^m$ if $k=n$.
If the strongly reflective form $F_n$ is non cuspidal of weight $n$, then
$F_n^m\,(dZ)^{\otimes m}$ cannot be extended to the compact model due to  cusp obstructions
($F_n^m$ should have zeroes of order at least $m$ along the boundary).
If $F_n$ is a cusp form of weight $k=n$ then we can consider $F_n$
as a cusp form with respect to the subgroup
$$
\Gamma_\chi=\ker(\chi\cdot\det)<\Gamma, \qquad
F_n\in S_n(\Gamma,\,\chi)<S_n(\Gamma_\chi,\,\det).
$$
Then  $F_{n}\,dZ$ is $\Gamma_\chi$-invariant
and, according to  Freitag's criterion, it  can be extended
 to a global section of the canonical bundle
$\Omega(\Bar{\cF}_L(\Gamma_\chi))$ for any smooth compact model $\Bar{\cF}_L(\Gamma_\chi)$
of ${\cF}_L(\Gamma_\chi)$.
Moreover the above consideration  with  Koecher's principle
shows that
any $m$-pluricanonical  form is equal, up to a constant, to $(F_{n}\,dZ)^{\otimes m}$.
Therefore in the last case of the theorem the strongly reflective cusp form
of canonical weight determines essentially the unique $m$-pluricanonical differential
form on $\cF_L(\Gamma_\chi)$.
\end{proof}
Some applications of this theorem will be given in the next section.

\section{Modular varieties  of Calabi-Yau type}
\label{sec2}

The problem of  constructing  a strongly reflective cusp form of canonical weight
(see the second case of Theorem \ref{thm-kdim}) is  far from trivial.
Note that any reflective modular form has a Borcherds product expansion
if the quadratic lattice is not very complicated (see \cite{Br}) but it is rather
difficult to construct Borcherds products of a fixed weight.
See, for example \cite{Ko} and \cite{GHS1} where cusp forms of canonical weight
were constructed for the moduli spaces of polarised $\Kthree$-surfaces.
Between those cusp forms there are no reflective modular forms.

There are only two examples of strongly reflective cusp forms of canonical weights
in the literature.
Both are Siegel modular forms of genus $2$
(i.e., the orthogonal group is of type $(2,3)$).
The first example is related to the strongly reflective modular form
$\Delta_1\in S_1(\Gamma_3,\chi_6)$ of weight $1$  (see \cite[Example 1.14]{GN4})
with respect to the paramodular group $\Gamma_3$
(the Siegel modular threefold $\Gamma_3\setminus\HH_2$ is the moduli
space of the $(1,3)$-polarised Abelian surfaces).
Then $\Delta_1(Z)^3 dZ$ is the unique  canonical differential form on
the Siegel threefold $(\ker \chi_6^3)\setminus\HH_2$ having
a Calabi--Yau model (see \cite{GH}). The second example is
the strongly reflective form $\nabla_3\in S_3(\Gamma_0^{(2)}(2), \chi_2)$
(see \cite{CG})
where $\Gamma_0^{(2)}(2)$ is the Hecke congruence subgroup of $\Sp_2(\ZZ)$
and $\chi_2: \Gamma_0^{(2)}(2)\to \{\pm 1\}$ is a binary character.
The Siegel cusp form $\nabla_3^2$ was first constructed  by T.~Ibukiyama
in \cite{Ib}. A Jacobi lifting and a Borcherds automorphic  product of $\nabla_3$
were given in \cite{CG}, where it was also  proved that the Kodaira dimension
of the Siegel threefold $(\ker \chi_2)\setminus\HH_2$ is equal to zero.
A Calabi-Yau model of this modular variety was founded  in
\cite{FS-M}.
Note that $(\ker \chi_2)\setminus\HH_2$ is a double cover of the rational Siegel
threefold $\Gamma_0^{(2)}(2)\setminus\HH_2$.
One of the main purposes  of this paper is to construct similar examples
for dimension larger than $3$.
\smallskip

Let $S$ be a positive definite lattice.
We put
$$
L(S)=2U\oplus S(-1), \qquad \sign(L(S))=(2,2+\rank S)=(2,2+n_0)
$$
where $S(-1)$ denotes
the corresponding negative definite lattice.
In the applications of this paper $S$ will be $D_n$, $mA_1$ or $mA_2$
where $mA_n=A_n\oplus \dots \oplus A_n$ ($m$ times).
We define two modular varieties
\begin{align}
\cS\cM(S)&=\Tilde{\SO}^+(L(S))
\setminus \cD(L(S)),\\
\cM(S)&=\Tilde{\Orth}^+(L(S))
\setminus \cD(L(S)).
\end{align}
Theorem \ref{thm-gt} shows that the main obstruction
to continuing of the pluri\-canonical
differential forms  on a smooth compact model of
a modular variety is its ramification divisor. In many case the ramification  divisor
of $\cS\cM(S)$ is strictly smaller than the ramification divisor of $\cM(S)$.
\begin{lemma}\label{lem-brd}
For odd $n\ge 3$ the ramification  divisor  of the projection
$$
\pi^+_{D_n}: \cD(2U\oplus D_n)\to \cS\cM(D_n)
$$
is defined by the reflections with respect to $(-4)$-vectors in $2U\oplus D_n$
with divisor $2$  where $\div_L(v)\ZZ=(v, L)$.
\end{lemma}
\begin{proof}
We recall that
$D_n$ is an even  sublattice  of the Euclidian lattice $\ZZ^n$
$$
D_n=\{(x_1,\dots,x_n)\in \ZZ^n\,|\, x_1+\dots+x_n\in 2\ZZ\}.
$$
We have $|D_n^\vee/D_n|=4$.
The discriminant form is
generated by the following four elements
$$
D_n^\vee/D_n=
\{\,0,\  e_n,\  (e_1+\dots+e_n)/2,\   (e_1+\dots+e_{n-1}-e_n)/2 \mod D_n\}.
$$
Then
$$
D_n^\vee/D_n\cong \ZZ/2\ZZ\times \ZZ/2\ZZ\quad (n\equiv 0 {\rm\  mod\  }2)
\quad{\rm or}\quad \ZZ/4\ZZ \quad (n\equiv 1{\rm\  mod\ } 2).
$$
Note that $(e_1+\dots+e_n)/2)^2= ((e_1+\dots+e_{n-1}-e_n)/2)^2=n/4$
is the minimal norm of the elements in the corresponding classes modulo
$D_n$.

For any even integral lattice $L$ with two hyperbolic planes
there is a simple description of the orbits of the primitive vectors.
According to the {\it Eichler criterion} (see \cite[page 1195]{G2} or \cite{GHS3}),
the $\Tilde{\SO}^+(L)$-orbit of any primitive $v\in L$ depends only on
$v^2=(v,v)$ and $v/ \div(v)\mod L$. We note that $v^*=v/ \div(v)$ is a primitive
vector in the dual lattice $ L^\vee$.

If $\sigma_v\in \Orth^+(L(D_n))$ then $v^2<0$ and  $\div(v)\mid v^2\mid 2\div(v)$.
Therefore, if $n$ is impair then  $\div(v)$ is a divisor of $4$  because
$4D_n^\vee <D_n$.
If $\div(v)=1$ then $v^2=-2$ and $-\sigma_v$ induces $-\id$ on the discriminant group.
(Note that  $\id=-\id$ on  $D_n^\vee/D_n$ for even $n$.)
If $\div(v)=2$  then $v^2=4$. It follows that
$v$ belongs to the $\Tilde{\SO}^+(L(D_n))$-orbit  of
one of  vectors  of type  $2e_i$ or $\pm e_{i_1}\pm e_{i_2}\pm e_{i_3} \pm e_{i_4}$ of square $4$.
Any   vector of type   $\pm e_{i_1}\pm e_{i_2}\pm e_{i_3} \pm e_{i_4}$  has divisor $1$
in $D_n$ for odd $n>4$.
If  $v=\pm 2e_i$ then $-\sigma_v$ induces identity on the discriminant group for odd $n$.
If $\div(v)=4$ and $v^2=-4$ or $-8$  then $(v/4)^2=-1/4$ or $-1/2$.
Both cases  are  impossible for odd $n$ because $(v^*)^2\equiv 0$, $1$ or $n/4\mod 2\ZZ$
for any $v^*\in D_n^\vee$.
\end{proof}

\begin{theorem}\label{thm-kdim0}
We have that $\cS\cM(D_7)$ is of general type and
$$
\kappa\bigl(\cS\cM(D_n))=
\begin{cases}
\ \ \ 0\quad &{\rm if\ }n=5,\\
-\infty\quad &{\rm if\ }n=3.
\end{cases}
$$
Moreover
$$
\kappa\bigl(\cS\cM(2A_2))=0\quad{\rm\  and\  }\quad
\kappa\bigl(\cS\cM^{(2)}(2A_1))=0
$$
where
$$
\cS\cM^{(2)}(2A_1)= \ker \chi_2\setminus\cD(L(2A_1))
$$
and $\chi_2: \Tilde{\SO}^+(2U\oplus 2A_1(-1))\to \{\pm 1\}$
is the  binary character  of the cusp  form  $\Delta_{4,\, 2A_1}$
from Theorem \ref{thm-liftA1}.
\end{theorem}
\begin{proof}
In the proof we use  strongly reflective cusp forms which will be constructed in
the next sections with the help of liftings of Jacobi modular forms.

1) The divisor of the cusp  form $\Lift(\psi_{12-n,\,D_{n}})$
of weight $12-n$
(see Theorem \ref{thm-liftD8} below) for $n=3$ and  $5$
is equal to the ramification  divisor of $\pi^+_{D_n}$. Therefore
we can  apply Theorem \ref{thm-kdim}. For $n=7$ the weight of
$\Lift(\psi_{5,\,D_{7}})$ is small. Therefore we can apply Theorem \ref{thm-gt}.

2) The modular variety $\cS\cM(2A_2)$ is of dimension $6$.
The strongly reflective cusp form $\Lift(\psi_{6,\,2D_{2}})$
(see Theorem \ref{thm-liftA2} below) has canonical weight.
The strongly reflective cusp form
$\Lift(\psi_{4,2A_{1}})\in S_4(\Tilde{\SO}^+(L(2A_1)),\chi_2)$
(see Theorem \ref{thm-liftA2}) is of weight $4$ which is the dimension of
$\cS\cM^{(2)}(2A_1)$.
Therefore we can apply the second part of Theorem \ref{thm-kdim}.
\end{proof}

\noindent
{\bf Remark 1.} {\it Varieties of Calabi--Yau type.}
We conjecture that each of the  three varieties of Kodaira dimension zero in
Theorem \ref{thm-kdim0}  have a Calabi--Yau model similar to  the two examples
mentioned in the beginning of \S 2.

\noindent
{\bf Remark 2.}
{\it Kodaira dimension of $\cM(S)$ in Theorem \ref{thm-kdim0}.}
Note that
$$
\cS\cM(S)=\Tilde{\SO}^+(L(S))
\setminus \cD(L(S))\to
\Tilde{\Orth}^+(L(S))
\setminus \cD(L(S))=\cM(S)
$$
is a covering of order $2$. The ramification divisor of
$\cM(S)$ contains all divisors of type $\cD_r(L(S))$ where
$r$ is any of  the  $(-2)$-vectors of $L(S)$. Analyzing   results of \cite{B4} or
using the  quasi pull-back of the Borcherds  form $\Phi_{12}$
(see a forthcoming paper of B. Grandpierre and V. Gritsenko
``{\it The baby functions of the Borcherds form} $\Phi_{12}$")
we can construct strongly reflective modular forms
for  $\Tilde{\Orth}^+(L(S))$ of big weight (see (\ref{wt-sb})) with respect
to $S=D_3$, $D_5$, $D_7$, $2A_2$ and $2A_1$.
Therefore using Theorem \ref{thm-kdim}
we obtain that for  all $S$ from Theorem \ref{thm-kdim0}
the modular variety $\cM(S)$ is of Kodaira dimension $-\infty$.
\smallskip

As we mentioned above  $\TS(L(S))\setminus \cD(L(S))$
is a double covering of the  modular variety $\cM(S)$
 which is the moduli space of the lattice polarised
$\Kthree$-surfaces with transcendence  lattice $T=L(S)$ (see \cite{N}, \cite{Do}).
Therefore $\cS\cM(D_5)$ can be considered as the moduli space of the lattice polarised
$\Kthree$-surfaces with transcendence  lattice $T=L(D_5)$ together with
a spin structure (a choice of orientation in $T$). See \cite[\S 5]{GHS1}
where the case of polarised $\Kthree$ surfaces of degree $2d$ with a spin structure
was considered. The Picard lattice  $\Pic(X_D)$ of a generic member $X_D$ of this moduli space
is
$$
\Pic(X_D)\cong (2U\oplus D_5)^\perp_{II_{3,19}}\cong U\oplus E_8(-1)\oplus A_3(-1)
$$
where $II_{3,19}=3U\oplus 2E_8(-1)\cong H^2(X, \ZZ)$ is the $\Kthree$-lattice.
The cases of $L(2A_2)$ and $L(2A_1)$ are similar
$$
\Pic(X_{2A_2})\cong U\oplus E_6(-1)\oplus E_6(-1)\cong U\oplus E_8(-1)\oplus A_2(-1)\oplus A_2(-1),
$$
$$
\Pic(X_{2A_1})\cong U\oplus E_7(-1)\oplus E_7(-1)\cong U\oplus E_8(-1)\oplus D_6(-1).
$$
The only difference here is that $\cS\cM^{(2)}(2A_1)$ is a double covering
of the moduli spaces $\cS\cM(2A_1)$
of the lattice polarised $\Kthree$-surfaces with a spin structure.

\section{Jacobi theta-series and the $D_n$-tower of strongly
reflective modular  forms}
\label{sec3}

We use  Jacobi modular forms in many variables, the corresponding
Jacobi lifting (\cite{G2}) and automorphic Borcherds products
(\cite{B1}, \cite{B3}, \cite{GN4}) in order to describe special
strongly reflective modular forms.
This will give us the proof of Theorem \ref{thm-kdim0}.
We see below that Jacobi forms (specially Jacobi theta-series) are sometimes more convenient
to use in our considerations than the corresponding vector-valued modular forms.

Let $L=2U\oplus S(-1)$ be an integral quadratic lattice of signature
$(2,n_0+2)$ where $U$ is the hyperbolic plane and $S$ is a positive definite
integral lattice of rank $n_0$ (then $S(-1)$ is negative definite).
The representation $2U\oplus S(-1)$ of $L$ gives us a choice of
a totally  isotropic plane  in $L$. It gives   the following tube
realization  $\cH_{2+n_0}$ of the type IV domain $\cD(L)$
$$
\cH_{2+n_0}=\{Z=(\omega, \gz, \tau)\in \HH^+\times (S\otimes\CC)\times \HH^+\,|\,
(\Im Z,\Im Z)_{U\oplus S(-1)}>0\}
$$
with
$
(\Im Z,\Im Z)=2\Im \tau\Im \omega -(\Im \gz,\Im \gz)>0
$
where $(\Im \gz,\Im \gz)$ is  the positive definite scalar product on $S$.
In the definition of Jacobi forms in many variables
we follow \cite{G2}.
\smallskip

\noindent
{\bf Definition}
A holomorphic (cusp or  weak) Jacobi form of weight $k$ and index $m$
with respect to $S$
($k \in \Bbb \ZZ$) is  a holomorphic function
$$
\phi: \HH^+\times (S\otimes \Bbb C)\to \CC
$$
satisfying the functional equations
\begin{align}\label{def-JF}
\phi(\frac{a\tau+b}{c\tau+d},\,\frac{\gz}{c\tau+d})&
=(c\tau+d)^{k}\exp\bigl (\pi i \frac{cm(\gz,\gz)}{c\tau+d}\bigr)
\phi(\tau ,\,\gz ),
\\\vspace{2\jot}
\phi(\tau,\gz+\lambda\tau+\mu)&
=\exp\bigl(-2\pi i \bigl(\frac{m}2 (\lambda,\lambda)\tau+m(\lambda,\gz)\bigr)\bigr)
\phi(\tau,\gz )
\end{align}
for any
$A=\begin{pmatrix}
a&b\\c&d
\end{pmatrix}\in\hbox{SL}_2(\Bbb Z)$ and any
$\lambda,\,\mu\in S$
and having a Fourier expansion
$$
\phi(\tau ,\,\gz )=
\sum_{n\in \ZZ,\ \ell \in S^\vee}
f(n,\ell)\,\exp\bigl(2\pi i (n\tau+(\ell,\gz) \bigr),
$$
where $n\ge 0$ for a weak Jacobi form,
$N_m(n,\ell)=2nm-(\ell,\ell)\ge 0$ for a holomorphic Jacobi form
and $N_m(n,\ell)>0$ for a Jacobi cusp form.
\smallskip

We denote the space of all holomorphic Jacobi forms by
$J_{k,m}(S)$. We use the  notation $J_{k,m}^{(cusp)}(S)$ and $J_{k,m}^{(weak)}(S)$
for the space of cusp and weak Jacobi forms.
If $J_{k,m}(S)\ne \{0\}$ then $k\ge \frac{1}2\rank S$ (see \cite{G1}).
The weight $k=\frac{1}2\rank S$ is called {\it singular}.
It is known (see \cite[Lemma 2.1]{G2}) that
$f(n,\ell)$ depends only the hyperbolic norm
$N_m(n,\ell)=2nm-(\ell,\ell)$ and the image of $\ell$
in the discriminant group $D(S(m))=S^\vee/mS$. Moreover,
$f(n,\ell)=(-1)^kf(n,-\ell)$.
\smallskip

\noindent{\bf Remark 1}.  {\it Fourier-Jacobi coefficients.}
Let  $F\in M_k(\TO(L))$. We consider its Fourier expansion in $\omega$
$$
F(Z)=f_0(\tau)+\sum_{m\ge 1}f_m(\tau, \gz)\exp{(2\pi i m\omega)}.
$$
where $f_0(\tau)\in M_k(\SL_2(\ZZ))$ and
$f_m(\tau, \gz)\in J_{k,m}(S)$.
The lifting construction of \cite{G1}--\cite{G2} defines a modular form
with respect to $\TO(L)$ with trivial character by its first
Fourier-Jacobi coefficient  from $J_{k,1}(S)$.
\smallskip

We note that  $J_{k,m}(S)=J_{k,1}(S(m))$
where $S(m)$ denotes the same lattice $S$ with the quadratic form multiplied by $m$,
and the space $J_{k,m}(S)$ depends essentially only on the discriminant form
of $S(m)$.   Any Jacobi form determines a vector valued modular form
related to the corresponding Weil representation (see \cite[Lemma 2.3--2.4]{G2}).
For $\phi\in J_{k,1}(S)$ we have
$$
\phi(\tau ,\,\gz)=\sum_{\substack{  n\in \ZZ,\ \ell \in S^\vee
\vspace{0.5\jot} \\
 2n-(\ell,\ell)\ge 0}}
f(n,\ell)\,\exp\bigl(2\pi i (n\tau+(\ell,\gz) \bigr)
=\sum_{h\in D(S) } \phi_h(\tau)\Theta_{S,h}(\tau, \gz),
$$
where $\Theta_{S,h}(\tau, \gz)$ is the Jacobi theta-series with characteristic
$h$ and the components of the vector valued modular forms
$(\phi_h)_{D(S)}$ have the following Fourier expansions at infinity:
$$
\phi_h(\tau)=\sum_{\substack{n\equiv -\frac{1}{2}(h,h)\mod \ZZ}}
f_h(r)
\exp{(2\pi i n\tau)}
$$
where
$f_h(n)=f(n+\frac{1}{2}(h,h),h)$.  This representation for a weak Jacobi form gives us
the next lemma

\begin{lemma}\label{lem-Jf}
Let $f(n,l)$ ($n\ge 0$, $\ell\in S^\vee$)
be a Fourier coefficient  of a weak Jacobi form
$\phi\in J_{k,1}^{(weak)}(S)$.  Then
$$
f(n,\ell)\ne 0\Rightarrow 2n-(\ell,\ell)\ge -\min_{v\in \ell+S} v^2.
$$
\end{lemma}

If $S=A_1=\langle2\rangle$ then $J_{k,m}(A_1)=J_{k,m}$
is the space of classical
holomorphic Jacobi modular forms studied in the book of M. Eichler and D. Zagier \cite{EZ}.
One more function, not mentioned in \cite{EZ}, is very important for
our considerations. This is the Jacobi theta-series
$$
\vartheta(\tau,z)=\vartheta_{11}(\tau ,z)=
\sum_{m\in \ZZ}\,\biggl(\frac{-4}{m}\biggr)\, q^{{m^2}/8}\,\zeta^{{m}/2}
\in J_{\frac{1}2,\frac{1}2}(v_\eta^3\times v_H)
$$
which is the  Jacobi form of weight $\frac{1}2$ and index $\frac 1{2}$
with multiplier system $v_\eta^3$ and the binary character $v_H$
of the Heisenberg group (see \cite[Example 1.5]{GN4}).
In the last formula we put $q=\exp{(2\pi i\tau)}$ and  $\zeta=\exp{(2\pi i z)}$,
the Kronecker symbol $\bigl(\frac{-4}{m}\bigr)=\pm 1$ if $m\equiv \pm 1\mod 4$
and is equal to zero for even $m$, $v_\eta$ is the multiplier system
of the Dedekind eta-function. The functional equation related to  the character
$v_H$ is
\begin{equation}\label{theta-H}
\vartheta(\tau ,z+\lambda \tau +\mu)=
(-1)^{\lambda+\mu}\exp{(-\pi i\,(\lambda^2 \tau  +2\lambda z))}\,
\vartheta(\tau , z)\quad (\lambda,\mu \in \ZZ).
\end{equation}
The multiplier system of $\vartheta(\tau ,z)$ is obtained from
the relation
$$
(2\pi i)^{-1}\frac{\partial\vartheta(\tau ,z)}{\partial z}\big|_{z=0}=
 \sum_{n>0}
\biggl(\frac{-4}{n}\biggr)nq^{n^2/8}
=\eta(\tau )^3.
$$
Therefore for any
$M=\left(\begin{smallmatrix}
 a&b\\c&d
\end{smallmatrix}\right) \in \SL_2(\ZZ)$ we have
\begin{equation}\label{theta-SL}
\vartheta\bigl(\frac{a\tau+b}{c\tau+d},\,\frac{z}{c\tau+d}\bigr)
=v_\eta^3(M)(c\tau+d)^{1/2}\exp(-\pi i \frac{cz^2}{c\tau+d})
\vartheta(\tau ,\,z ).
\end{equation}
Moreover
\begin{equation}\label{theta-JTF}
\vartheta(\tau ,\,z)=
-q^{1/8}\zeta^{-1/2}\prod_{n\ge 1}\,(1-q^{n-1} \zeta)(1-q^n \zeta^{-1})(1-q^n)
\end{equation}
and $\vartheta(\tau ,\,z)=0$ if and only if $\tau=\lambda \tau+\mu$
($\lambda,\mu \in \ZZ$) with multiplicity $1$.

The Jacobi modular forms related to $\vartheta(\tau ,\,z)$ were very important
in \cite{GN4} for the construction of reflective Siegel  modular forms and
the corresponding Lorentzian Kac-Moody superalgebras. The next example shows
the role of $\vartheta(\tau ,\,z)$ in the context of this paper.
\smallskip

\noindent
{\bf Example}. {\it Jacobi form of singular weight for $D_8$.}
Let us put
\begin{equation*}
\psi_{4,\,D_8}(\tau, (z_1,\dots,z_8))=\vartheta(\tau ,\,z_1)\vartheta(\tau ,\,z_2)\cdot \dots\cdot
\vartheta(\tau ,\,z_8)\in J_{4,1}(D_8).
\end{equation*}
This is a Jacobi form of singular weight for $D_8$.
The functional equations  (\ref{theta-H}) and (\ref{theta-SL}) give us
the equations (\ref{def-JF})--(6).
Using the Dedekind $\eta$-function we can define Jacobi forms for
any $D_k$. Let $2\le k\le 8$. We put
\begin{equation}\label{eq-Dk}
\psi_{12-k,\,D_k}(\tau, \gz_k)=
\eta(\tau)^{24-3k}\ \vartheta(\tau ,\,z_1)\dots
\vartheta(\tau ,\,z_k)\in J_{12-k,1}(D_k).
\end{equation}
Similar Jacobi forms we have for any $k>8$. The construction depends only
on $k$ modulo  $8$.
The function $\psi_{12-k,\,D_k}$ vanishes with order one for $z_i=0$
($1\le i\le k$).
Using the Jacobi lifting from \cite{G2}  we obtain a modular form
$\Lift(\psi_{12-k,\,D_k})$ of  weight $12-k$ with respect to
$\Tilde\Orth^+(2U\oplus D_k(-1))$.
The form $\psi_{12-k,\,D_k}$ is the first Fourier-Jacobi coefficient
of $\Lift(\psi_{12-k,\,D_k})$.
The lifting  preserves the divisor of $\psi_{12-k,\,D_k}$
but the lifted form has, usually, some additional divisors.
Therefore we see that using a Jacobi form of type (\ref{eq-Dk}) we obtain
a modular form with respect to an orthogonal group whose divisor contains
the  union of the translations of the  rational quadratic divisors defined
by equation $z_i=0$ ($1\le i\le k$).
This example gives a good illustration of why the language
of Jacobi forms is very useful for our considerations.

\begin{theorem}\label{thm-liftD8}
For $2\le k\le 8$
$$
\Delta_{12-k, D_k}=\Lift(\psi_{12-k,\,D_{k}})\in M_{12-k}(\Tilde{\Orth}^+(L(D_k)))
$$
is strongly reflective.
More exactly, if $k\ne 4$ then
$$
\div_{\cD(L(D_k))} \Lift(\psi_{12-k,\,D_{k}})=
\bigcup_{\substack {\pm v\in L(D_k)\vspace{0.5\jot} \\  v^2=-4,\ \div(v)=2}}
\cD_v(L(D_k)),
$$
where all vectors $v$ in the last union belong to the same $\Tilde{\Orth}^+(L(D_k))$-orbit.
The divisor for   $k=4$ is defined by the orbit of $2e_1\in D_4$.
If $k<8$ then $\Lift(\psi_{12-k,\,D_{k}})$ is  a cusp form.
\end{theorem}

\noindent
{\bf Remark 1.}
$\Lift(\psi_{4,\,D_{8}})$  and $\Lift(\psi_{7,\,D_{5}})$
are  strongly reflective modular form of singular and canonical  weight
respectively.
The modular group of the lifting is, in fact, larger.
For $k\ne 4$ we have
\begin{equation}\label{eq-maxmodgr}
\Lift(\psi_{12-k,\,D_{k}})\in M_{12-k}(\Orth^+(L(D_k)),\tilde{\chi})
\end{equation}
where the binary character $\tilde{\chi}$ of ${\Orth}^+(L(D_k))$
is defined by the relation
$$
\tilde\chi(g)=1\  \Leftrightarrow\
g|_{L(D_k)^\vee/L(D_k)}=\id.
$$
If $k=4$ then the maximal modular group of $\Lift(\psi_{8,\,D_{4}})$
is a subgroup of order $3$ in $\Orth^+(L(D_4))$ (see  the proof of the theorem).

\begin{proof}
We described  $D_k$ and $D_k^\vee/D_k$ in the proof of Lemma \ref{lem-brd}.
In particular we see that  $2D_8^\vee < D_8$.
We denote the ``half-integral" part of $D_8^\vee$
by
\begin{equation}\label{eq-D8odd}
D_8^\vee(1)=\{(e_1+\dots+e_7\pm e_8)/2+D_8\}.
\end{equation}
The Fourier expansion  of $\psi_{4,\,D_8}$ has the following form
$$
\psi_{4,\,D_8}(\tau, \gz_8)=
\sum_{\substack{  n\in \ZZ,\ \ell \in D_8^\vee
\vspace{0.5\jot} \\
 2n-(\ell,\ell)= 0}}
f(n,\ell)\,\exp\bigl(2\pi i (n\tau+(\ell,\gz_8) \bigr)=
$$
$$
=\sum_{\ell\in D_8^\vee(1)}
\,\biggl(\frac{-4}{2\ell}\biggr)
\exp(\pi i \bigl((\ell,\ell)\tau+2(\ell,\gz_8)\bigr))
$$
where
\begin{equation}\label{eq-kron}
\biggl(\frac{-4}{2\ell}\biggr)=
\biggl(\frac{-4}{2l_1}\biggr)\dots \biggl(\frac{-4}{2l_8}\biggr).
\end{equation}
According to \cite[Theorem 3.1]{G2}
$$
\Lift(\psi_{12-k,D_k})(Z)\in M_{12-k}(\Tilde{\Orth}^+(L(D_k)))
$$
is a modular form  with trivial character.
The lifting of a Jacobi form $\phi_k(\tau,\gz)\in J_{k, 1} (S)$ (with $f(0,0)=0$)
of weight $k$ was defined in \cite{G2} by the formula
$$
\Lift(\phi_k)(\tau, \gz_n,\omega)=
\sum_{m\ge 1}m^{-1}
\bigl(\phi_k(\tau,\gz_n)e^{2\pi i m\omega}\bigr)\mid_k T_{-}(m)
$$
\begin{equation}\label{eq-litfJ}
=\sum_{m\ge 1}
m^{-1}\sum_{\substack{ad=m\\ b \mod d}}
\ a^{k}\phi_k\bigl(\frac{a\tau+b}d,\  a\gz_n\bigr)\,e^{2\pi i m\omega}.
\end{equation}
According to (\ref{eq-litfJ})
\begin{equation}\label{eq-Fexplift}
\Lift(\phi_k)(Z)=
\sum_{\substack{ n,m>0,\,\ell\in S^\vee\\
\vspace{0.5\jot} 2nm-(\ell,\ell)\ge 0}}\
\sum_ {d|(n,\ell,m)}
d^{k-1}f(\frac{nm}{d^2},\frac{\ell}d)\,e(n\tau+(\ell,\gz)+m\omega)
\end{equation}
where $d|(n,\ell,m)$ denotes a positive integral  divisor of the vector
in $U\oplus S^\vee(-1)$.
For example we  can calculate the Fourier expansion of the modular form
of singular weight
$$
\Lift(\psi_{4,D_8})(Z)=\sum_{\substack{
 n,\,m\in \ZZ_{>0}\\ \vspace{0.5\jot}
\ell=(l_1,\dots,l_8)\in D_8^\vee\\
\vspace{0.5\jot}  2nm-(\ell,\ell)=0}}
\ \sum_{\substack{t=2^w\\ \vspace{0.5\jot}
(n,\,\ell,\,m)/t\in U\oplus D_8^\vee(1)}}
t^3\sigma_3((n/t,\ell/t,m/t))
$$
\begin{equation}\label{eq-FexpD8}
\biggl(\frac{-4}{2\ell/t}\biggr)
\exp(2\pi i (n\tau+ (\ell,\gz_8)+m\omega))
\end{equation}
where $t=t(n,\ell,m)$ is the maximal common power of $2$ in $(n,\ell, m)$,
i.e. $t=2^w$, $n/t$ and $m/t$ are integral,   $\ell/t\in  D_8^\vee(1)$.
Then the greatest  common divisor
$d=(n/t,\ell/t,m/t)$ of a vector in $U\oplus D_8(-1)$
is an odd number and $\sigma_3(d)=\sum_{a|d}a^3$.

The maximal modular group of
$\Lift(\psi_{12-k,\,D_{k}})$ is, in fact, larger than $\Tilde{\Orth}^+(L(D_k))$.
The orthogonal group $\Orth(D_k^\vee/D_k)$ of the finite discriminant group
is of order $2$ for any $k\not\equiv 4 \mod 8$.
If $k\equiv 4 \mod 8$ then $\Orth(D_k^\vee/D_k)\cong S_3$.
The lifting is anti-invariant under the transformation
$z_1\to -z_1$ (the reflection $\sigma_{e_1}$) inducing
the non-trivial automorphism of $\Orth(D_k^\vee/D_k)$ if $k\ne 4$.
Therefore, if $k\ne 4$ then $\Lift(\psi_{12-k,\,D_{k}})$
is a modular form with respect to $\Orth^+(L(D_k))$ with  the character $\tilde{\chi}$
defined in  (\ref{eq-maxmodgr}).
If $k=4$ then the permutation of the coordinates give us
only the  permutation of the classes of
$(e_1+e_2+e_3+e_4)/2$ and $(e_1+e_2+e_3-e_4)/2$ in $D_4^\vee/D_4$.
Therefore $\Lift(\psi_{8,D_{4}})$
is modular with a binary character for  a subgroup of index $3$
in $\Orth^+(L(D_4))$.

To show that $\Lift(\psi_{12-k,D_{k}})$ is strongly reflective we consider  its Borcherds
product expansion. We can  construct the Borcherds product
using a Jacobi form of weight $0$ in a way similar
to \cite[Theorem 2.1]{GN4}.
We can obtain a  weak Jacobi form of weight $0$ for $D_8$
using the so-called ``minus" Hecke operator $T_{-}(2)$
(see (\ref{eq-litfJ}) and \cite{G2}, page 1193).
We put
$$
\phi_{0,\, D_8}(\tau,\gz)=
\frac{2^{-1}\psi_{4,D_8}|_4\,T_{-}(2)}{\psi_{4,D_8}}=
$$
$$
8\prod_{i=1}^8
\frac{\vartheta(2\tau,2z_i)}{\vartheta(\tau,z_i)}
+\frac{1}2
\prod_{i=1}^8
\frac{\vartheta(\frac{\tau}2,z_i)}{\vartheta(\tau,z_i)}
+\frac{1}2
\prod_{i=1}^8
\frac{\vartheta(\frac{\tau+1}2,z_i)}{\vartheta(\tau,z_i)}.
$$
Then $\phi_{0,\, D_8}$ is a weak Jacobi form of weight $0$ and index $1$. Using the Jacobi
product formula (\ref{theta-JTF}) we obtain
$$
\phi_{0,\, D_8}(\tau,\gz_8)=\sum_{n\ge 0,\ \ell \in D_8^\vee}
c(n,\ell)\,\exp\bigl(2\pi i (n\tau+(\ell,\gz_8) \bigr)=
$$
$$
\zeta_1^{\pm 1}+\dots +\zeta_8^{\pm 1}+8+q(\dots)\qquad {\rm where\ }\  \zeta_i=\exp(2\pi i z_i).
$$
We noted above that the Fourier coefficient  $c(n,\ell)$
of weak Jacobi form $\phi_{0,\, D_8}$ depends only on the hyperbolic norm $2n-\ell^2$ and
the class $\ell \mod D_8^\vee$. Moreover, if $c(n,\ell)\ne 0$ then
$2n-(\ell,\ell)\ge -2$
(see Lemma \ref{lem-Jf} and the representation of $D_8^\vee/D_8$ above).
According to the Eichler criterion, the primitive vectors in $2U\oplus D_8^\vee$
with norm equal to $-1$ and $-2$ form three $\Tilde{\SO}^+(L(D_8))$-orbits
represented by the elements of the minimal norms in $D_8^\vee$.
Therefore the $q^0$-term in the Fourier expansion of $\phi_{0,\, D_8}$
contains all types of the Fourier coefficients
$\phi_{0,\,D_8}$ with $2n-(\ell,\ell)<0$.
Using the Borcherds product construction as in  Theorem 2.1
of \cite{GN4} we  obtain a modular form
$$
B(\phi_{0,\,D_8})(Z)=
q(\zeta)^{(1/2,\dots,1/2)}s\prod_{\substack{n,m\ge 0\\ \vspace{0.5\jot} (n,\ell,m)>0}}
(1-q^n(\zeta)^\ell s^m)^{c(nm,\ell)}
$$
$$
=\bigl(\psi_{4,D_8}(\tau, \gz_8)e^{2\pi i \omega}\bigr)
\exp\biggl(-\phi_{0,D_8}(\tau,\gz_8)|\sum_{m\ge 1} m^{-1}T_{-}(m)e^{2\pi i m\omega}\biggr)
$$
in  $M_4(\Tilde{\Orth}^+(L(D_8)))$  with the trivial character
of $\Tilde{\Orth}^+(L(D_8))$
where $q=\exp(2\pi i \tau)$, $s=\exp(2\pi i \omega)$
and $(\zeta)^\ell=(\zeta_1^{l_1}\zeta_2^{l_2}\dots \zeta_8^{l_8})$.
Its  divisors are  determined by the Fourier coefficients $\zeta_i^{\pm 1}$,
i.e.  by the  vectors $\pm e_i\in D_8$ ($1\le i \le 8$).
According to the Eichler criterion
the $\Tilde{\SO}^+(L(D_8))$-orbit of any vector $v\in 2U\oplus D_8(-1)$ with
$v^2=-4$ and $\div(v)=2$ is defined by $v/2 \mod D_8$.
Therefore any such vector belongs to the orbit of $e_1$. We proved that
$$
\div_{\cD(L(D_8))} B(\phi_{0,\,D_8})(Z)=
\bigcup_{\substack {\pm v\in L(D_8)\vspace{0.5\jot} \\  v^2=-4,\ \div(v)=2}}
\cD_v.
$$
The modular projection of this divisor on
$\cS\cM(D_8)$ is irreducible.
The  formula (\ref{eq-litfJ})  shows that the lifting preserves
the divisor of  type $z_i=0$
of the Jacobi form. It follows that the divisor
of $\Lift(\psi_{4,D_8})(Z)$ contains the divisor given in  Theorem \ref{thm-liftD8}.
 According to Koecher's principle
$$
\Lift(\psi_{4,\,D_8})(Z)=B(\phi_{0,\,D_8})(Z)
$$
because they have the same first Fourier-Jacobi coefficient.

We can also use the weak Jacobi form $\psi_{4,D_8}$ in order to construct
Borcherds  products for the lattices $2U\oplus D_k(-1)$ with $2\le k \le 8$.
We put
$$
\phi_{0,\,D_k}(\tau,\gz_k)=
\phi_{0,\,D_8}(\tau,\gz_k,0,\dots,0)=\zeta_1^{\pm 1}+\dots \zeta_k^{\pm 1}+(24-2k)+q(\dots).
$$
Then $\phi_{0,\,D_k}(\tau,\gz_k)$ is a weak Jacobi form of weight $0$ for $D_k$.
Using the same arguments as  for $D_8$ above we obtain
$$
\Lift(\psi_{12-k,\,D_k})(Z)=B(\phi_{0,\,D_k})(Z).
$$
If $k=4$ then the divisor of the last function is smaller than
the divisor in Theorem \ref{thm-liftD8}. It is defined by the
$\Tilde{\Orth}^+(L(D_4))$-orbit of $2e_1$ (see the proof of Lemma \ref{lem-brd}).

If $k<8$ then
the Jacobi form $\psi_{12-k,\,D_k}$ is cuspidal. Therefore its lifting
is a cusp forms because $D_k$ is a maximal even lattice.
\end{proof}

Let  $\gG(F)$ be   a Lorentzian Kac--Moody (super) Lie algebra
of Borcherds type determined by an automorphic form $F$.
Note that  the generators and relations of this algebra are defined
by the Fourier expansion of $F$ at a~zero-dimensional cusp.
The Borcherds product of $F$ determines only the multiplicities of the positive
roots of this algebra. Therefore for  an explicit construction
of  $\gG(F)$ one has to find the Fourier expansion of $F$ at a cusp.
This explains the importance of explicit formulae for the Fourier coefficients.
We give the Fourier expansion of $\Delta_{12-k,\, D_k}$
in Corollary \ref{cor-Fexp} below (see also  (\ref{eq-FexpD8})
and (\ref{eq-FexpD7})).

In 1996 (see \cite{B2}) Borcherds constructed the strongly reflective
automorphic discriminant $\Phi_4$ of the moduli space on Enriques surfaces
$$
\Phi_4\in M_4(\Orth^+(U\oplus U(2)\oplus E_8(-2)),\,\chi_2)
$$
where $\chi_2$ is a binary character.  The Borcherds products of   $\Phi_{4}$
were given  in two non-equivalent  cusps (see  \cite[Example 13.7]{B3}).
This function is called sometimes Borcherds-Enriques form.
See \cite{HM} for its  applications
in string theory.
In the next corollary we obtain a Jacobi lifting  construction of $\Phi_4$.

\begin{corollary}\label{cor-Phi4}
The form $\Lift(\psi_{4,\,D_8})$ is equal, up to a constant, to
the Borcherds modular form $\Phi_4$.
\end{corollary}
\begin{proof}
The divisor of $\Phi_4$ in $\cD(L_E)$
with $L_E=U\oplus U(2)\oplus E_8(-2)$
is determined by the $\Orth^+(M_E)$-orbit of a $(-2)$-vector $v\in U$
(see \cite{B2}). Note that a renormalization of $L$
($L\to L(n)$) does not change the orthogonal group
and $\Orth(L)=\Orth(L^\vee)$ for any lattice $L$.
Therefore
$$
\Orth^+(L_E)=\Orth^+(L_E^\vee(2))=\Orth^+(U(2)\oplus U\oplus E_8(-1))
\cong \Orth^+(L(D_8))
$$
because
$$
U(2)\oplus E_8(-1)\cong  U\oplus D_8.
$$
These two  hyperbolic lattices  correspond to the two different $0$-dimensional cusps
of the modular variety $\Orth^+(L_E)\setminus \cD(L_E)$.
An arbitrary  $(-2)$-vector of $L_E$ becomes  a reflective vector of $L_E^\vee(2)$
of length $-4$. Therefore the $(-2)$-reflective divisor
of $\Orth^+(L_E)\setminus \cD(L_E)$ corresponds to the
$(-4)$-reflective divisor of $\Orth^+(L(D_8))\setminus \cD(L(D_8))$.
We see that the modular forms $\Phi_4$ and $\Lift(\psi_{4,\,D_8})$ have the same divisor.
Therefore they are equal, up to a constant, according to Kocher's principle.
\end{proof}

The automorphic Borcherds products  related to  the quasi pull-backs of $\Phi_4$
appear in the new paper   \cite{Y} of K.-I. Yoshikawa.
These modular forms $\Phi_V$ are the automorphic discriminants of the K\"ahler moduli
of a Del Pezzo surface $V$ of degree $1 \le n\le 9$ (compare with \cite{GN5}).
The function $\Phi_V$ determines also the analytic torsion of some exceptional
Calabi--Yau threefolds of Borcea-Voisin type (see \cite[Theorem 1.1]{Y}).

\begin{corollary}\label{cor-Y}
Let $V$ be a Del Pezzo surface of degree $1\le \deg V\le 6$.
The modular form $\Phi_{V}$ of Yoshikawa  is equal, up to a constant,
to the modular form
$\Delta_{4+\deg V,\, D_{8-\deg V}}=\Lift(\psi_{4+\deg V,\,D_{8-\deg V}})$
of Theorem \ref{thm-liftD8}.
\end{corollary}
\begin{proof}
The proof of Theorem \ref{thm-liftD8} shows that the singular modular form
$\Lift(\psi_{4,\,D_8})(Z)$  is the  generating function for
the $D_8$-towers of the strongly reflective modular forms
of Theorem \ref{thm-liftD8}. We have
$$
(2\pi i)^{-1}\frac{\partial \psi_{4,\,D_8}(\tau,\gz_8)}{\partial z_8}\big|_{z_8=0}=
\psi_{7,\,D_7}(\tau,\gz_7).
$$
Therefore
$\Lift(\psi_{5,D_7})$ is the quasi pull-back  (see \cite[pp. 200--201]{B1},
and \cite[\S 6]{GHS1}) of $\Lift(\psi_{4,\,D_8})$ along the divisor $z_8=0$.
We can continue this process. Then
 $\Lift(\psi_{6,\,D_6})$ is  the quasi pull-back of
$\Lift(\psi_{5,\,D_7})$ along $z_7=0$ and so on till $\Lift(\psi_{10,\,D_2})$.

The Yoshikawa modular forms $\Phi_V$  for $\deg V\ge  1$ also constitute a  similar tower
with respect to the quasi pull-backs based on the Borcherds form $\Phi_4$
 (see \cite[\S 6]{Y}).
To finish the proof we use again Koecher's principle.
\end{proof}

Theorem \ref{thm-liftD8} and (\ref{eq-litfJ})  give the formula for the  Fourier expansion of
$\Lift(\psi_{12-k,\,D_k})$ and $\Phi_V$.
For any integral $3m>0$ we put
$$
\eta(\tau)^{3m}=\sum_{n>0} \tau_{3m}(\frac{n}{8}) \,q^{n/8}.
$$
We put $\tau_0(n)=1$ if and only if $n=1$.
\begin{corollary}\label{cor-Fexp}
We have the following Fourier expansion
$$
\Lift(\psi_{12-k,\,D_k})(Z)=\sum_{\substack{
 n,\,m\in \ZZ_{>0}\\ \vspace{0.5\jot}
\ell=(l_1,\dots,l_n)\in D_k^\vee}}
\sum_{d\mid (n,\ell,m)}
d^{11-k}
\tau_{24-3k}(\frac{2nm-\ell^2}{2d^2})
$$
$$
\biggl(\frac{-4}{2\ell/d}\biggr)
\exp(2\pi i (n\tau+ (\ell,\gz_k)+m\omega))
$$
where  $d\mid (n,\,\ell,\,m)$ is an integral  divisor
in $U\oplus D_k(-1)^\vee$ and we use the notation (\ref{eq-kron}).
\end{corollary}
\begin{proof}
In the notation above we have
$\psi_{12-k,\,D_k}(\tau, \gz_k)=$
$$
\sum_{n\in \NN,\,\ell \in D_k^\vee(1)}
\,\tau_{24-3k}(\frac{2n-\ell^2}2)\biggl(\frac{-4}{2\ell}\biggr)
\exp(2\pi i (n\tau+(\ell,\gz_k)).
$$
Using (\ref{eq-Fexplift}) we obtain the formula of the corollary.
The Kronecker symbol  $\bigl(\frac{-4}{2l_i/d}\bigr)=0$ if
\ $2l_i/d$ is pair. Therefore the vector $\ell/d$ belongs in fact
to the odd part of $D_k^\vee$ (see (\ref{eq-D8odd})) and
we can divide any vector by its maximal common power of $2$
in $(n,\ell,m)$.  Then we make a summation on
odd common divisors like in  (\ref{eq-FexpD8}).

We note that $\eta(\tau)^3$ has elementary formula for its Fourier coefficients:
$\tau_3(n/8)=\bigl(\frac{-4}{N}\bigr)N$ if and only if $n=N^2$.
Therefore we  have
$$
\Lift(\phi_{5,\,D_7})(\tau,\gz_7,\omega)=
\sum_{N\ge 1}
\sum_{\substack{
2nm-\ell^2=\frac{N}4\\
\vspace{0.5\jot}
n,m\in \NN,\,\ell \in D_7^\vee}}
\sum_{d|(n,\ell,m)}
$$
\begin{equation}\label{eq-FexpD7}
N
\,\biggl(\frac{-4}{2\ell/d}\biggr)\,
\exp(2\pi i (n\tau+(\ell,\gz_7)+m\omega)).
\end{equation}
\end{proof}

The modular forms $\Phi_V$ determine some automorphic Lorentzian Kac-Moody super Lie
algebras related to Del Pezzo surfaces.
The formula of Corollary \ref{cor-Fexp} gives us the generating function
of the imaginary simple roots of these algebras.
\smallskip

\noindent
{\bf Remark 2.} The quasi pull-back of
$\Delta_{10,\,D_2}=\Phi_V$ ($\deg V=6$) is  the Siegel  modular form
$\Delta_{11}\in S_{11}(\Gamma_2)$ (see \cite[(3.11)]{GN4})
where $\Gamma_2$ is the paramodular group of  type $(1,2)$.
Note that  (see \cite[Lemma 1.9]{GN4})
$$
\Gamma_2/\{\pm \id\}\cong \Tilde{\SO}^+(L(D_1))\qquad
{\rm where }\ \ D_1=\langle 4\rangle.
$$
Then $\Delta_{11}
\left(\begin{smallmatrix}
\tau&z\\z&\omega\end{smallmatrix}\right)
=\Lift(\eta(\tau)^{21}\vartheta(\tau, 2z))$
is strongly reflective (see  \cite[Example 1.15 and (3.11)]{GN4}).
Its divisor contains two irreducible components $\{z=0\}$
and $\{z=1/2\}$.
The quasi pull-back of the Siegel modular form $\Delta_{11}$
along $z=0$ is equal to the product of the  Ramanujan modular forms
$\Delta_{12}(\tau)\Delta_{12}(2\omega)$ (in the  $\Gamma_2$-coordinates).

\section{$A_2$-tower of strongly
reflective modular  forms}

We start with a useful general fact.
\begin{lemma}\label{lem-oplus}
Let $\phi_i(\tau,\gz_i)\in J_{k_i,m}(S_i)$
where $S_i$ is a positive definite  integral lattice.
Then
$$
\phi_1(\tau,\gz_1)\cdot\phi_2(\tau,\gz_2)
\in J_{k_1+k_2, m}(S_1\oplus S_2).
$$
\end{lemma}
The proof of the lemma follows directly from the definition.
\smallskip

In order to construct Jacobi modular forms for $D_n$ we used
the Jacobi form $\vartheta(\tau,z)$ which is the denominator function
of the affine Lie algebra $\hat A_1$.
In this section we use
the denominator function
of the affine Lie algebra $\hat A_2$.
Let $A_2=\ZZ a_1+\ZZ a_2$ where $a_1$ and $a_2$ are the simple roots of $A_2$.
We can rewrite this lattice in
the Euclidian basis $(e_1,e_2,e_3)$
$$
A_2\otimes \CC=z'_1e_1+z'_2e_2+z'_3e_3
$$
where
$$
z'_1=z_1,\quad z'_2=z_2-z_1,\quad z'_3=-z_2.
$$
The denominator function of the affine Kac--Moody algebra  $\hat A_2$
is associated to the holomorphic Jacobi form
of singular weight $1$ with character $v_\eta^8$ of order $3$
$$
\Theta(\tau,z_1,z_2)=
\frac{1}{\eta(\tau)}
\vartheta(\tau,z_1)\vartheta(\tau,z_2-z_1)\vartheta(\tau,z_2)
\in J_{1,1}(A_2; v_\eta^8)
$$
(see \cite{Ka} and \cite{De}).
Therefore using Lemma \ref{lem-oplus}
we can define three holomorphic Jacobi forms with trivial character
\begin{align*}
\psi_{9,\,A_2}(\tau,z_1,z_2)&=
\eta^{16}(\tau)\Theta(\tau,z_1,z_2)
\in J^{(cusp)}_{9,1}(A_2),\\
\psi_{6,\,2A_2}(\tau,z_1,\dots,z_4)&=
\eta^{8}(\tau)\Theta(\tau,z_1,z_2)\Theta(\tau,z_3,z_4)
\in J^{(cusp)}_{6,1}(2A_2),\\
\psi_{3,\,3A_2}(\tau,z_1,\dots,z_6)&=
\Theta(\tau,z_1,z_2)\Theta(\tau,z_3,z_4)
\Theta(\tau,z_5,z_6)
\in J_{3,1}(3A_2).
\end{align*}
The Jacobi lifting construction gives
three modular forms of orthogonal type.
In particular  we obtain
$$
\Lift(\psi_{3,\,3A_2})\in M_1(\Tilde{\Orth}^+(L(3A_2)))
$$
of singular weight with trivial character. For $n=2$ we obtain a cusp form of canonical weight.
The divisor of the Jacobi form induces a divisor
of the lifting.
Note that $z_i=0$ is the hyperplane of the reflection
$\sigma_{\lambda_i}$ where $\lambda_i$ is  a fundamental weight
of $A_2$ and
$$
A_2^\vee/A_2=\{\,0,\,\lambda_1,\,\lambda_2 \mod A_2\}
$$
where the fundamental weights  are vectors of minimal length $\lambda_i^2=2/3$
in the corresponding $A_2$-classes.
Then $3\lambda_i\in A_2(-1)$ is reflective $(-6)$-vector,
i.e. one of the $G_2$-roots of the lattice $A_2$.
\begin{theorem}\label{thm-liftA2}
Let $k=1$, $2$ or $3$.
Then
$$
\Delta_{12-3k,\,kA_2}=\Lift(\psi_{12-3k,\,kA_{2}})\in M_{12-3k}(\Tilde{\Orth}^+(L(kA_2)))
$$
is strongly reflective and
$$
\div_{\cD(L(kA_2))}\Lift(\psi_{12-3k,\,kA_{2}})=
\bigcup_{\substack {\pm v\in L(kA_2)\vspace{0.5\jot} \\  v^2=-6,\ \div(v)=3}}
\cD_v(L(kA_2)).
$$
For $k=1$ and $2$ the lifting is a cusp form.
\end{theorem}
\begin{proof}
The proof is similar to the proof of Theorem \ref{thm-liftD8}.
It is enough to find a Borcherds product for the singular modular form
$\Lift(\psi_{3,\,3A_2})$.
To construct a weak Jacobi form of weight $0$ we again use the Hecke
operator $T_{-}(2)$ of the  Jacobi lifting. We put
$$
\phi_{0,\,3A_2}(\tau,\gz)=
\frac{2^{-1}\psi_{3,\,3A_2}|_3 T_{-}(2)}{\psi_{3,\,3A_2}}
$$
where
$$
2^{-1}\psi_{3,\,3A_2}|\, T_{-}(2)=
4\psi_{3,\,3A_2}(2\tau,2\gz_6)
+\frac{1}2
\psi_{3,\,3A_2}(\frac{\tau}2,\gz_6)
+\frac{1}2
\psi_{3,\,3A_2}(\frac{\tau+1}2,\gz_6).
$$
Analyzing the divisor of $\psi_{3,\,3A_2}$ we see that $\phi_{0,\,3A_2}\in J^{(week)}_{0,1}(3A_2)$.
Moreover the direct calculation shows that
$$
\phi_{0,\,3A_2}(\tau,\gz)=
\sum_{\substack{n\ge 0\\ \ell\in U\oplus 3A_2(-1)^\vee}}
c(n,\ell)\exp(2\pi i (n\tau+(\ell,\gz_6)))=
$$
$$
6+\sum_{i=1,3,5}
(\zeta_i^{\pm 1}+\zeta_{i+1}^{\pm 1} + (\zeta_i\zeta_{i+1}^{-1})^{\pm 1})
+q(\dots)
$$
where $\zeta_i=\exp(2\pi i z_i)=\exp(2\pi i (\gz_6,\lambda_i))$
and $\lambda_i$ are the fundamental weights of the corresponding copies of $A_2$.
According to  Lemma \ref{lem-Jf}, in order to  obtain
all Fourier coefficients  $c(n,\ell)\ne 0$ with $2n-\ell^2<0$
one has to check only coefficients with $2n-\ell^2\ge -2$.
The sum over $i\in \{1,2,3\}$ in the formula above contains  all such coefficients.
Therefore the Borcherds product $B(\phi_{0,\,3A_2})$ is of weight $c(0,0)/2=3$
with divisors of order $1$ along all $\Tilde{\Orth}^+(L(3A_2))$-orbits
of the $(-6)$-vectors $\pm \lambda_i$, $\pm \lambda_{i+1}$ and
$\pm(\lambda_i-\lambda_{i+1})$ ($i\in \{1,2,3\}$).
Using Koecher's principle as in the proof of Theorem \ref{thm-liftD8}
we find that
$$
\Lift(\psi_{3,\,3A_2})=B(\phi_{0,\,3A_2}).
$$
Therefore $\Lift(\psi_{3,\,3A_2})$ is strongly reflective.
To find a weak Jacobi for $A_2$ and $2A_2$ we put
$$
\phi_{0,\,2A_2}(\tau,\gz_4)=\phi_{0,\,3A_2}(\tau,\gz_4,0,0)=
12+\sum_{i=1,3}(
\zeta_i^{\pm 1}+\zeta_{i+1}^{\pm 1} + (\zeta_i\zeta_{i+1}^{-1})^{\pm 1})+q(\dots)
$$
and
$$
\phi_{0,\,A_2}(\tau,\gz_2)=\phi_{0,\,3A_2}(\tau,\gz_2,0,0,0,0)=
18+
\zeta_1^{\pm 1}+\zeta_{2}^{\pm 1} + (\zeta_1\zeta_{2}^{-1})^{\pm 1}+q(\dots).
$$
Therefore
$$
\Lift(\psi_{6,\,2A_2})=B(\phi_{0,\,2A_2})\quad{\rm and }\quad
\Lift(\psi_{9,\, A_2})=B(\phi_{0,\,A_2}).
$$
The last  modular form of weight $9$ was constructed
in \cite[Proposition 4.2]{De}. The method of  construction of $\phi_{0,\,A_2}$ in \cite{De}
was different.
\end{proof}
\noindent
{\bf Remark.} The strongly reflective form $\Lift(\psi_{12-3n,\,nA_{2}})$
(for $n=2$, $3$)
is invariant with respect to permutations of any two copies of $A_2$.
A permutation does not change the branch  divisor. It follows that
the cusp form $\Lift(\psi_{6,\,2A_{2}})$ determines, in fact, two modular varieties
of Kodaira dimension $0$. These are  the variety of Theorem 2.2
and the variety for the double extension of the stable orthogonal group
$\Tilde{\SO}^+(2U\oplus 2A_2(-1))$ corresponding to the permutation
of the two copies of $A_2(-1)$.

\section{$A_1$-tower of strongly
reflective modular  forms}

In this section we use the Jacobi theta-series as Jacobi modular forms
of half-integral index.
\begin{theorem}\label{thm-liftA1}
There exist four  strongly reflective modular forms for $L(nA_1)$
with  $n=1$, $2$, $3$ and $4$:
$$
\Delta_{6-n,\, nA_1}\in M_{6-n}(\Tilde{\SO}^+(L(nA_1)),\chi_2)
$$
where $\chi_2: \Tilde{\SO}^+(L(nA_1))\to \{\pm 1\}$.
Moreover
$$
\div_{\cD(L(nA_1))} \Delta_{6-n,\, nA_1}=
\bigcup_{\substack {\pm v\in L(nA_1)\vspace{0.5\jot} \\  v^2=-2,\ \div(v)=2}}
\cD_v(L(nA_1))
$$
and $\Delta_{6-n,\, nA_1}$ is a cusp form if $n<4$.
\end{theorem}
\begin{proof}
In the proof we construct these modular forms as Borcherds products.
As in the proof of Theorem \ref{thm-liftD8} and Theorem \ref{thm-liftA2},
the main function of this $4A_1$-tower is the modular form
$\Delta_{2,\, 4A_1}$  of singular weight.
We put
$$
\psi_{2,\,4A_1}(\tau, \gz_4)=
\vartheta(\tau,z_1)\vartheta(\tau,z_2)\vartheta(\tau,z_3)\vartheta(\tau,z_4).
$$
This is a Jacobi form of weight $2$ and index $\frac{1}2$ with character $v_\eta^{12}\times v_H$
of order $2$ where $v_H$ is the binary character of the Heisenberg group
$H(4A_1)$. We can define the following weak Jacobi form of weight $0$
$$
\phi_{0,\,4A_1}(\tau,\gz_4)=
\frac{3^{-1}\psi_{2,\,4A_2}|_2 T_{-}(3)}{\psi_{2,\,4A_1}}
\in J_{0, 1}^{(weak)}(4A_1)
$$
where
$$
3^{-1}\psi_{2,\,4A_1}|_2T_{-}(3)=
3\psi_{3,\,3A_2}(3\tau,3\gz_4)
+\frac{1}3
\sum_{b=0}^{2}\psi_{2,\,4A_1}(\frac{\tau+b}3,\gz_4).
$$
The straightforward  calculation shows that
$$
\phi_{0,\,4A_1}(\tau,\gz_4)=
4+\sum_{i=1}^4
\zeta_i^{\pm 1} +q(\dots).
$$
We put $\tilde\phi_{0,\,4A_1}(Z)=\phi_{0,\,4A_1}(\tau,\gz_4)e^{2\pi i \omega}$.
In terms of this Jacobi form the Borcherds product is given by the following formula
(see \cite[(2.7)]{GN4})
$$
B(\phi_{0,\,4A_1})(\tau,\gz_4,\omega)=
\bigl(\psi_{2,\,4A_1}(\tau,\gz_4)e^{\pi i \omega}\bigr)
\exp\biggl(-\sum_{m\ge 1}m^{-1} \tilde\phi_{0,4A_1}|T_{-}(m)(Z)\biggr).
$$
This formula shows that
$\Delta_{2,\, 4A_1}=B(\phi_{0,\,4A_1})$ is a modular form of weight $2$
with respect to $\Tilde{\SO}^+(L(4A_1))$
and  with divisor described in the theorem.
The coefficient before the exponent is the first Fourier-Jacobi coefficient of
$B(\phi_{0,\,4A_1})$.
Therefore   the character of $B(\phi_{0,\,4A_1})$ is the binary character
induced by the character of $\psi_{2,\,4A_1}(\tau,\gz_4)$.
As in the proof of Theorem \ref{thm-liftD8} and Theorem \ref{thm-liftA2}
of this section  we  put
$$
\phi_{0,\,nA_1}(\tau,\gz_n)=\phi_{0,\,4A_1}(\tau,\gz_n, 0,\dots, 0)\qquad
(1\le n<4).
$$
It gives  the three other strongly reflective modular forms.
The last  function   of this $4A_1$-tower is the Siegel modular form
$\Delta_5$, which is  the Borcherds product
defined by the Jacobi form
$$
\phi_{0,\,A_1}(\tau,z)=\phi_{0,1}(\tau,z)=\zeta+10+\zeta^{-1}+q(\dots)
\in J_{0,1}^{(weak)}
$$
(see \cite{GN1}).
\end{proof}

It is possible to get a Jacobi lifting construction of the strongly
reflective modular forms of the last theorem. We can prove that
$$
\Delta_{6-n,\,nA_1}=\Lift\bigl(\eta(\tau)^{12-3n}\prod_{i=1}^{n}\vartheta(\tau,z_i)\bigr).
$$
Here we take  a Jacobi lifting with a character similar way
to  \cite[Theorem 1.12]{GN4}. This lifting  gives the elementary formula for the
Fourier coefficients of  $\Delta_{2,\,4A_1}$ and $\Delta_{3,\,3A_1}$
similar to (\ref{eq-FexpD8}). For example we have the following Fourier expansion
of  the modular form of singular weight
$$
\Lift(\psi_{2,\,4A_1})(Z)=
\sum_{\substack{ \ell=(l_1,\dots, l_4)\\
 \vspace{0.5\jot} l_i\equiv \frac 1{2} \,{\rm mod \,}\ZZ}}
$$
$$
\sum_{\substack{
 n,\,m\in \ZZ_{>0}\\
 \vspace{0.5\jot} n\equiv m\equiv 1\,{\rm mod\,}\ZZ\\
\vspace{0.5\jot}  nm-(\ell,\ell)=0}}
\sigma_1((n,\ell,m))
\biggl(\frac{-4}{2l_1}\biggr)\dots \biggl(\frac{-4}{2l_4}\biggr)
\exp(\pi i (n\tau+ 2(\ell,\gz_4)+m\omega)).
$$
See details in the forthcoming paper of
F.~Cl\'ery and V.~Gritsenko ``{\it Jacobi modular forms and root systems}".
\smallskip

\noindent
{\bf Remark.} The $14$ strongly reflective modular forms constructed in
Theorems \ref{thm-liftD8},  \ref{thm-liftA2} and \ref{thm-liftA1}
determine Lorentzian Kac--Moody super Lie algebras
of Borcherds type in a way described in \cite{GN1}--\cite{GN4}.
The details of this construction and many other examples will appear in our
forthcoming paper with V.~Nikulin.
\smallskip

\noindent
{\bf Conclusion.}
To finish this paper we  would like to  characterize the three series of
the strongly reflective modular forms  considered above and to formulate
a conjecture on similar modular forms.
To this aim we come back to the first two examples of \S 1.
The divisor of the Borcherds form $\Phi_{12}$ is defined by all $(-2)$-roots
in $II_{2,26}$. This is  the irreducible  reflective divisor in
$\cF_{II_{2,26}}(\Orth^+(II_{2,26}))$.
For the Igusa modular forms the situation is different.
The divisor of $\Delta_{35}$ is  the branch divisor
of the Siegel modular threefold
$$
\pi :\HH_2\to \Sp_2(\ZZ)\setminus \HH_2\cong
\cF_{2U\oplus A_1(-1)}(\SO^+(2U\oplus A_1(-1)))
$$
containing  two irreducible components.
The first one $\pi(\cD_{-2}(1))$ is defined  by the $(-2)$-vectors
with divisor $1$.
The second one  $\pi(\cD_{-2}(2))$ is  generated   by
the $(-2)$-vectors with divisor $2$.
They are the Humbert modular surfaces of determinant $4$
and $1$ respectively.
The divisor of the Igusa form
$\Delta_5\left(\begin{smallmatrix}
\tau&z\\z&\omega \end{smallmatrix}\right)
=\Lift(\eta(\tau)^9\vartheta(\tau,z))$ coincides with
$\pi(\cD_{-2}(2))$.
This is the simplest divisor of the Siegel threefold $\pi(\{z=0\})$.
The modular form  $\Delta_{35}$ is not a Jacobi lifting. Its first Fourier-Jacobi
coefficient is zero and the second one is equal to
$\eta(\tau)^{69}\vartheta(\tau,2z)$. See \cite{GN2} where  the  Borcherds product
of $\Delta_{35}$ was constructed. Moreover
$\Delta_{35}$ can be considered as a ``baby" function of $\Phi_{12}$
(see our  forthcoming paper mentioned in Remark 2 of \S 3).
We may say that the  fourteen strongly  reflective modular forms constructed in \S 3--\S 5
are similar to $\Delta_5$. Each of them  is the Jacobi lifting of its first Fourier--Jacobi
coefficient and its divisor is the simplest divisor of the corresponding
modular variety.
For the modular forms of the $4A_1$-tower, the simplest  divisor is  generated by
the $(-2)$-vectors  with divisor $2$.
For the $D_8$-tower, the  divisor is  generated by
the $(-4)$-reflective vectors, and
for the $3A_2$-tower, it is  generated by
the $(-6)$-reflective vectors of divisor $3$.
All these divisors  are complementary  to  the divisor defined by
the $(-2)$-roots with divisor one.

We remind  the following general  fact.
Let  $L$ be  a non-degenerate even integral  lattice and $h\in L$ be
a primitive vector with $h^2=2d$. If $L_h$ is the orthogonal complement of $h$ in
$L$ then
$$
|\det L_h|=\frac {|2d|\cdot |\det L|}{\div (h)^2}.
$$
Therefore, $\det L_h$  for  the reflective vectors considered above
is smaller than $\det L_r$ for a $(-2)$-root  $r$ with $\div(r)=1$.

There is the second explanation why these divisors are simpler.
The divisor $\pi(\cD_v)$, where $\pi$ is a modular projection, is a modular variety of orthogonal type.
For reflective vectors ($\sigma_v$ or $-\sigma_v$ is in $\Gamma$) they form
the reflective obstruction  to  extending
of pluricanonical forms to a compact model (see the proof of Theorem 1.5).
A numerical measure of this  obstruction is given by  the Hirzebruch--Mumford volume
of $\pi(\cD_v)$. This volume was calculated explicitly in \cite{GHS4}
for arbitrary   indefinite lattice.
We can say that the divisor  $\pi( \cD_v)$ is simpler if
its  Hirzebruch--Mumford volume is smaller.

We would like to formulate a conjecture related to  the modular forms of this paper.
Let $L$ be a lattice of signature $(2,n)$ with $n\ge 3$. We assume that
the branch divisor of the modular variety $\cF_L(\Gamma)$ has several components.
We suppose that there exists a strongly reflective modular forms $F$ whose
divisor is equal to the simplest reflective divisor. {\it We conjecture that
$F$ could be constructed as an additive (Jacobi)
lifting.}

Note that the existence of a strongly  reflective modular form implies a strong
condition on the lattice (see \cite{GN3}). The Weyl group $W$ of the hyperbolic root system
related to the simplest divisor should be arithmetic (elliptic or parabolic in the sense of
\cite{GN3}) and the root system admits a Weyl vector.
\smallskip

{\noindent {\bfseries Acknowledgements}:
This work was supported by  the grant ANR-09-BLAN-0104-01.
The results of this  paper were presented at the conferences in Kyoto (June, 2009)
in Edinburgh (September 2009) and at ``Moduli in Berlin" (September 2009).
I  would  like to thank S.~Konda, V. Nikulin,  N.~Scheithauer and K.-I. Yoshikawa
for useful conversations.
The author is  grateful to  the  Max-Planck-Institut f\"ur Mathematik in Bonn
for support and for providing excellent
working conditions.}

\smallskip
\noindent
V.A.~Gritsenko\\
Universit\'e Lille 1\\
Laboratoire Paul Painlev\'e\\
F-59655 Villeneuve d'Ascq, Cedex\\
France\\
{\tt Valery.Gritsenko@math.univ-lille1.fr
\end{document}